\documentclass{amsart}
\usepackage{geometry}
\newtheorem{theorem}{Theorem}[section]
\newtheorem{lemma}[theorem]{Lemma}
\newtheorem{proposition}[theorem]{Proposition}
\newtheorem{corollary}[theorem]{Corollary}
\theoremstyle{definition}
\newtheorem{definition}[theorem]{Definition}

\newcommand{\ns}{{\mathbb N\,}}

\newcommand{\hs}{{\mathcal H\,}}
\newcommand{\ks}{{\mathcal K\,}}

\theoremstyle{remark}
\newtheorem{remark}[theorem]{Remark}

\numberwithin{equation}{section}
\allowdisplaybreaks[4]

\begin{document}

\title{On weaving g-frames for Hilbert spaces}

\author{Dongwei Li}
\address{School of Mathematical Sciences, University of Electronic Science and Technology of China, 611731, P. R. China}
\email{ldwmath@163.com}

\author{Jinsong Leng}
\address{School of Mathematical Sciences, University of Electronic Science and Technology of China, 611731, P. R. China}
\email{jinsongleng@126.com}
\author{Tingzhu Huang}
\address{School of Mathematical Sciences, University of Electronic Science and Technology of China, 611731, P. R. China}
\email{tingzhuhuang@126.com}
\author{Xiaoping Li}
\address{School of Mathematical Sciences, University of Electronic Science and Technology of China, 611731, P. R. China}
\email{lixiaoping.math@uestc.edu.cn}




\begin{abstract}
Weaving frames  are powerful tools in wireless sensor networks and pre-processing signals. In this paper, we introduce the concept of weaving for g-frames in Hilbert spaces. We first give some properties of weaving g-frames and present two necessary conditions in terms of frame bounds for weaving g-frames. Then we study the properties of weakly woven g-frames and give a sufficient condition for weaving g-frames. It is shown that weakly woven is equivalent to woven. Two sufficient conditions for weaving g-Riesz bases are given. And a weaving equivalent of an unconditional g-basis for weaving g-Riesz bases is considered. Finally, we present Paley-Wiener-type perturbation results for weaving g-frames.
\end{abstract}

\maketitle

\section{Introduction}
The concept of frame in  Hilbert space introduced by Duffine and Schaeffer in their work on nonharmonic Fourier series \cite{duffin1952class}, reintroduced in 1986 by Daubechies, et al. \cite{daubechies1986painless} and popularized from then on.  Frames have established themselves by now as a standard notion in applied mathematics and engineering. Nice properties of frames have made them useful in functional analysis \cite{han2008frame,sun2015}, filter bank theory \cite{kovacevic2002filter}, coding theory \cite{leng2011optimall,leng2011optimal}, probability
statistics \cite{ehler2012random,li2016some}, and  quantum
information \cite{klyachko2008simple}.

Throughout this paper, let $\hs$ and $\ks$ be two Hilbert spaces and $\{\hs_i\}_{i\in I}$ be a sequence of closed subspaces of $\ks$, where $I$ is a subset of $\ns$. Let $L(\hs,\hs_i)$ be the collection of all bounded linear operators from $\hs$ into $\hs_i$. We denote by $I_{\hs}$ the identity operator on $\hs$. For $T\in L(\hs)$, we denote $T^{\dagger}$ for pseudo-inverse of $T$. Let $$[m]=\{1,2,\cdots,m\}~~{\rm and}~~[m]^c=\ns\setminus[m]=\{m+1,m+2,\cdots\}$$
for a given natural number $m$.
\begin{definition}
A family of vectors $\{f_i\}_{i\in I}$ in a Hilbert space $\hs$ is said to be a frame if there are constants $0<A\le B<\infty$ such that, for every $f\in\hs$, 
\begin{equation}
A\|f\|^2\le\sum_{i\in I}\left| \left\langle f,f_i\right\rangle \right| ^2\le B\|f\|^2,
\end{equation}
where $A$ and $B$ are lower frame bound and upper frame bound, respectively.
\end{definition}
 A frame is called a tight frame if $A=B$, and is called a Parseval frame if $A=B=1$. If a sequence $\{f_i\}_{i\in I}$ satisfies the upper bound condition in (1.1), then $\{f_i\}_{i\in I}$ is also called a Bessel sequence. 

Recently, several generalizations of frames in Hilbert space have been proposed, for example, fusion frames \cite{casazza2008fusion}, pseudo-frames \cite{li2004pseudoframes},  oblique frames \cite{fornasier2004quasi} and outer frames \cite{aldroubi2004wavelets}, and it was shown that they have important applications.  Sun in \cite{sun2006g} introduced the concept of g-frame and proved that all of the above generalizations of frames are special cases of g-frames.
G-frames in Hilbert spaces have been studied intensively with the development of kinds of applications. For the connection between the theory of g-frames and quantum theory as in \cite{han2011operator}. 
\begin{definition}
	A sequence $\{\Lambda_i\}_{i\in I}\subset L(\hs,\hs_i)$ of bounded operators from $\hs$ to $\hs_i$ is said to be a generalized frame, or simply a g-frame, for $\hs$ with respect to $\{\hs_i\}_{i\in I}$ if there are two positive constants $A$ and $B$ such that
	\begin{equation}
	A\|f\|^2\le \sum_{i\in I}\|\Lambda_if\|^2\le B\|f\|^2,~~~\forall f\in\hs.
	\end{equation}
\end{definition}
We call $A$ and $B$ the lower and upper g-frame bounds, respectively. We call $\{\Lambda_i\}_{i\in I}$ a tight g-frame if $A=B$ and a Parseval g-frame if $A=B=1$. If only the second inequality of (1.2) is required, we call it a g-Bessel sequence with bound $B$.

We say simply a g-frame for $\hs$ whenever the space sequence $\{\hs_i\}_{i\in I}$ is clear.

We say $\{\Lambda_i\}_{i\in I}$ is a g-frame sequence, if it is a g-frame for  ${\rm span}\{\Lambda^*_i(\hs_i)\}_{i\in I}$.

We say  $\{\Lambda_i\}_{i\in\ns}$ is an unconditional g-sequence in $\hs$ if and only if there is a constant $\gamma>0$ so that for all $\sigma\subset\ns$ and for all  $\{g_i\}_{i\in\ns}\in\oplus_{i\in\ns}\hs_i$ we have
$$\|\sum_{i \in  \sigma}\Lambda^*_ig_i\|\le \gamma\|\sum_{i \in  \sigma}\Lambda^*_ig_i+\sum_{i \in  \sigma^c}\Lambda^*_ig_i\|=\gamma\|\sum_{i \in  \ns}\Lambda^*_ig_i\|.$$
We say   $\{\Lambda_i\}_{i\in I}$ is g-complete if ${\rm\overline{span}}\{\Lambda^*_i(\hs_i):i\in I\}=\hs$.
\begin{definition}
A sequence $\{\Lambda_i\}_{i\in I}$ is called a g-Riesz basis for $\hs$ if it is g-complete and there exist constants $0<C\le D<\infty$ such that for any finite index set $J\subset I$, any $g_i\in\hs_i$ we have
$$C\sum_{i \in  J}\|g_i\|^2\le \|\sum_{i \in  J}\Lambda^*_ig_i\|^2\le D\sum_{i \in  J}\|g_i\|^2.$$
\end{definition}
We say $\{\Lambda_i\}_{i\in I}$ is g-orthonormal basis for $\hs$ if it satisfies 
$$\left\langle \Lambda^*_ig_i,\Lambda^*_jg_j\right\rangle =\delta_{ij}\left\langle g_i,g_j\right\rangle,~~{\rm and}~~\sum_{i \in  I}\|\Lambda_if\|^2=\|f\|^2$$
for any $g_i\in\hs_i$, $g_j\in\hs_j$ and $f\in\hs$.

For each sequence $\{\hs_i\}_{i\in I}$, we define the space $\oplus_{i\in I}\hs_i$ by 
$$\oplus_{i\in I}\hs_i=\{\{f_i\}_{i\in I}|f_i\in\hs_i,\|\{f_i\}_{i\in I}\|^2_2=\sum_{i\in I}\|f_i\|^2<\infty\},$$
with the inner product defined by 
$$\left\langle \{f_i\},\{g_i\}\right\rangle =\sum_{i\in I}\left\langle f_i,g_i\right\rangle .$$

The synthesis operator of $\{\Lambda_i\}_{i\in I}$ is given by 
$$T_{\Lambda}:\oplus \hs_i\longrightarrow\hs;~~~T_{\Lambda}\{g_i\}_{i\in I}=\sum_{i\in I}\Lambda_i^*g_i,~~~{\rm for ~all}~g_i\in\hs_i.$$

We call the adjoint of $T_{\Lambda}$ the analysis operator which is given by $T^*_{\Lambda}f=\{\Lambda_if\}_{i\in I}$.

By composing $T_{\Lambda}$ and $T^*_{\Lambda}$, we obtain the g-frame operator
$$S_{\Lambda}f=T_{\Lambda}T^*_{\Lambda}f=\sum_{i\in I}\Lambda^*_i\Lambda_i f$$
which is bounded, positive and invertible. Then, the following reconstruction formula takes place for all $f\in\hs$
$$f=S^{-1}_{\Lambda}S_{\Lambda}f=S_{\Lambda}S^{-1}_{\Lambda}f=\sum_{i\in I}\Lambda^*_i\Lambda_iS^{-1}_{\Lambda}f=\sum_{i\in I}S^{-1}_{\Lambda}\Lambda^*_i\Lambda_if.$$
We call $\{\Lambda_iS^{-1}_{\Lambda}\}_{i\in I}$ the canonical dual g-frame of $\{\Lambda_i\}_{i\in I}$. 

In \cite{sun2006g}, Sun showed that every g-frame can be considered as a frame. More precisely, let $\{\Lambda_i\}_{i\in I}$ be a g-frame for $\hs$ and $\{e_{i,j}\}_{j\in J_i}$ be an orthonormal basis for $\hs_i$, then there exists a frame $\{u_{i,j}\}_{i\in I,j\in J_i}$ of $\hs$ such that 
\begin{equation}
u_{i,j}=\Lambda^*_ie_{i,j},
\end{equation}
and
$$\Lambda_if=\sum_{j\in J_i}\left\langle f,u_{i,j}\right\rangle e_{i,j},~~~\forall f\in\hs,$$
and 
$$\Lambda^*_ig_i=\sum_{j\in J_i}\left\langle g_i,e_{i,j}\right\rangle u_{i,j},~~~\forall g_i\in\hs_i.$$
We call $\{u_{i,j}\}_{i\in I,j\in J_i}$ the frame induced by $\{\Lambda_i\}_{i\in I}$ with respect to  $\{e_{i,j}\}_{i\in I,j\in J_i}$. The next lemma is a characterization of g-frame by a frame.
\begin{lemma}\cite{sun2006g}
	Let $\{\Lambda_i\}_{i\in I}$ be a family of linear operators and $u_{i,j}$ be defined as in (1.3). Then $\{\Lambda_i\}_{i\in I}$ is a g-frame (tight g-frame, g-Riesz basis) for $\hs$ if and only if $\{u_{i,j}\}_{i\in I,j\in J_i}$ is a frame (tight frame, Riesz basis) for $\hs$.
\end{lemma}

Recently,  Bemrose, Casazza, Gr\"{a}chenig, Lammers and Lynch in \cite{bemrose2015weaving} introduced a new concept of weaving frames which is motivated by a problem regarding distributed signal processing.
For example, in wireless sensor network where frames may be subjected to distributed processing under different frames. For given two frames $\{f_i\}_{i\in I}$ and $\{g_i\}_{i\in I}$, we can think of each $i\in I$ as a sensor or node. And for each one we measure a signal with either $f_i$ or $g_i$, so that the collected information is the set of numbers $\{\left\langle f,f_i\right\rangle \}_{i\in\sigma}\cup\{\left\langle f,g_i\right\rangle \}_{i\in\sigma^c}$ for some subset $\sigma\subset I$. If $\{f_i\}_{i\in\sigma}\cup\{g_i\}_{i\in\sigma^c}$ is a frame for any choice of $\sigma\subset I$, $f$ can still be recovered robustly from these measurements, no matter which kind of measurement has been made at each node.  Hence, weaving frames have potential applications in wireless sensor networks that require distributed processing under different frames and possibly in the preprocessing of signals using Gabor frames.  Many interesting and useful results  of weaving frames are obtained, we refer to \cite{casazza2016weaving,casazza2015weaving,vashisht2017continuous,vashisht2016weaving} as references for those.

In a large wireless sensor network, due to the restrictions of hardware conditions and power, the network should be split into some sub-networks. Since stable space splittings are equivalent to g-frames \cite{sun2006g}, for each sub-network,  we measure a signal with either $\Lambda_i$ or $\Gamma_i$, where $\{\Lambda_i\}_{i\in I}$ and $\{\Gamma_i\}_{i\in I}$ are two g-frames. Then the collected information is the set of numbers $\{\Lambda_{i}f\}_{i\in\sigma}\cup\{\Gamma_{i}f\}_{i\in\sigma^c}$ for some subset $\sigma\subset I$. Can stable of the signal $f$ be obtained regardless of which measurement is taken? That is, is $\{\Lambda_{i}\}_{i\in\sigma}\cup\{\Gamma_{i}\}_{i\in\sigma^c}$ a g-frame for any choice of $\sigma$ ?

In this paper, we extend the concept of weaving frames to weaving g-frames for Hilbert spaces. We develop the fundamental properties of weaving g-frames for their own sake. When we only require each weaving  to be a g-frame without uniform lower bound and upper bound, we introduce a seemingly notation, weakly woven g-frames.
We also present Paley-Wiener-type perturbation results for weaving g-frames.

Let us briefly describe the concept of weaving frames in  Hilbert spaces.

\begin{definition}
A family of frames $\{f_{ij}\}_{i\in I,j\in[m]}$ for a Hilbert space $\hs$ is said to be woven if there are universal constants $A$ and $B$ so that for every partition $\{\sigma_j\}_{j\in[m]}$ of $I$, the family $\{f_{ij}\}_{i\in \sigma_j,j\in[m]}$ is a frame for $\hs$ with lower and upper frame bounds $A$ and $B$, respectively. Each family $\{f_{ij}\}_{i\in \sigma_j,j\in[m]}$ is called a weaving.
\end{definition}
If we only require each weaving to be a frame which is not necessarily with uniform bounds $A$ and $B$, we say that woven is weakly.
\begin{definition}
	A family of frames $\{f_{ij}\}_{i\in \ns,j\in[m]}$ in $\hs$ is said to be weakly woven if for every partition $\{\sigma_j\}_{j\in[m]}$ of $\ns$, the family $\{f_{ij}\}_{i\in \sigma_j,j\in[m]}$ is a frame for $\hs$. 
\end{definition}
The authors in \cite{bemrose2015weaving} proved that weakly woven is equivalent to the frames being woven.
\section{Weaving g-frames}
We first give the concept of weaving g-frames.
\begin{definition}
A family of g-frames $\{\Lambda_{ij}\}_{i\in I,j\in[m]}$ for a Hilbert space $\hs$ is said to be woven if there are universal constants $A$ and $B$ so that for every partition $\{\sigma_j\}_{j\in[m]}$ of $I$, the family $\{\Lambda_{ij}\}_{i\in \sigma_j,j\in[m]}$ is a g-frame for $\hs$ with lower and upper frame bounds $A$ and $B$, respectively. 
\end{definition}
As in the case of discrete weaving frame, weaving g-frame automatically has a universal upper frame bound.
\begin{proposition}
	If each $\Lambda_i=\{\Lambda_{ij}\}_{i\in I}$ is a g-Bessel sequence for $\hs$ with bounds $B_j$ for all $j\in [m]$, then every weaving is a g-Bessel sequence with $\sum_{j=1}^{m}B_j$ as a Bessel bound.
\end{proposition}
\begin{proof}
	Let $\{\sigma_j\}_{j\in[m]}$ be any partition of $I$. Then, for every $f\in\hs$, we have
\begin{equation*}
\sum_{j=1}^M\sum_{i\in \sigma_j}\left\| \Lambda_{ij}f\right\| ^2\le \sum_{j=1}^M\sum_{i\in I}\left\| \Lambda_{ij}f\right\| ^2\le \sum_{j=1}^MB_j\|f\|^2.
\end{equation*}	
This gives the desired result.
\end{proof}
We now show that a bounded operator applied to woven g-frames leaves them woven.
\begin{proposition}
Let $\{\Lambda_{ij}\}_{i\in I,j\in[m]}$ be a woven family of g-frames for $\hs$ with common frame bounds $A$ and $B$. If $T$ has a close range on $\hs$, then $\{\Lambda_{ij}T\}_{i\in I,j\in[m]}$ is also woven with bounds 
$A\|T^{\dagger}\|^{-2}$ and $B\|T\|^2$. 
\end{proposition}

\begin{proof}
It is known that if a g-frame  has bounds $A$ and $B$, then applying an bounded operator $T$ with close range to it gives a new g-frame with bounds $A\|T^{\dagger}\|^{-2}$ and $B\|T\|^2$. Since the sequence  $\{\Lambda_{ij}\}_{i\in \sigma_j,j\in[m]}$  is a g-frame with lower and upper bounds $A$ and $B$, respectively, for any partition $\{\sigma_j\}_{j\in[m]}$ of $I$, the sequence $\{\Lambda_{ij}T\}_{i\in \sigma_j,j\in[m]}$ is also a g-frame with bounds $A\|T^{\dagger}\|^{-2}$ and $B\|T\|^2$. That is, $\{\Lambda_{ij}T\}_{i\in I,j\in[m]}$ is woven with universal bounds $A\|T^{\dagger}\|^{-2}$ and $B\|T\|^2$.
\end{proof}
\begin{remark}
	If $T$ is an invertible operator in Proposition 2.3, then $\{\Lambda_{ij}T\}_{i\in I,j\in[m]}$ is woven with universal bounds $A\|T^{-1}\|^{-2}$ and $B\|T\|^2$. In particular, the bounds do not change if $T$ is unitary.
\end{remark}
\begin{corollary}
	Let $\{\Lambda_{ij}\}_{i\in I,j\in[m]}$ be a woven family of g-frames for $\hs$ with common frame bounds $A$ and $B$. If $S_{\Lambda}^{(j)}$ represents the frame operator of  $\{\Lambda_{ij}\}_{i\in I}$ for each $j\in [m]$, then the canonical dual  $\{\Lambda_{ij}(S_{\Lambda}^{(j)})^{-1}\}_{i\in I,j\in[m]}$ is also woven with bounds 
	$\frac{1}{B}$ and $\frac{1}{A}$. 
\end{corollary}
The next result gives condition on multiplying the g-frame elements by individual constants and still be left with woven g-frames.
\begin{theorem}
Let $\{\Lambda_{ij}\}_{i\in I,j\in[m]}$ be a woven family of g-frames for $\hs$ with common frame bounds $A$ and $B$. Let $0\le C\le |a^{(j)}_i|^2\le D<\infty$ for all $j\in[m]$, then $\{a^{(j)}_i\Lambda_{ij}\}_{i\in I,j\in[m]}$ is also woven with bounds $AC$ and $BD$. 
\end{theorem}
\begin{proof}
Since the sequence  $\{\Lambda_{ij}\}_{i\in \sigma_j,j\in[m]}$  is a g-frame with lower and upper bounds $A$ and $B$, respectively, for any partition $\{\sigma_j\}_{j\in[m]}$ of $I$, we have
\begin{equation*}
AC\|f\|^2=\min{|a^{(j)}_i|^2}A\|f\|^2\le\sum_{j=1}^{m}\sum_{i\in \sigma_j}\|a^{(j)}_i\Lambda_{ij}f\|^2\le \max{|a^{(j)}_i|^2}B\|f\|^2=BD\|f\|^2,
\end{equation*}
yielding the desired bound.
\end{proof}
\begin{remark}
If $|a^{(1)}_i|=|a^{(2)}_i|=,\cdots,=|a^{(m)}_i|=\frac{1}{A}$ in Theorem 2.5, then  $\{\frac{1}{A}\Lambda_{ij}\}_{i\in I,j\in[m]}$ is also woven with bounds $1$ and $B/A$. 
\end{remark}
The following proposition gives that weaving may possibly be check on subindex set of original.
\begin{proposition}
Let $J\subset I$. If a family of g-frames $\{\Lambda_{ij}\}_{i\in J,j\in[m]}$ is woven, then $\{\Lambda_{ij}\}_{i\in I,j\in[m]}$ is also woven.
\end{proposition}
\begin{proof}
For any $\sigma_j\subset I$, then $\sigma_j\cap J\subset J$. Let $A$ be the lower bound of $\{\Lambda_{ij}\}_{i\in \sigma_j\cap J,j\in[m]}$,  then  for any $f\in\hs$  we have
\begin{align*}
A\|f\|^2&\le \sum_{j=1}^m\sum_{i\in \sigma_j\cap J}\|\Lambda_{ij}f\|^2 \le \sum_{j=1}^m\sum_{i\in \sigma_j}\|\Lambda_{ij}f\|^2.
\end{align*}
Since $\{\Lambda_{ij}\}_{i\in I}$ is a g-Bessel sequence for all $j\in[m]$ for $\hs$, from Proposition 2.2, the upper bound of $\{\Lambda_{ij}\}_{i\in I,j\in[m]}$ is always given. This implies $\{\Lambda_{ij}\}_{i\in I,j\in[m]}$ is woven for $\hs$.
\end{proof} 
The Proposition 2.8 shows that adding elements to $\{\Lambda_{ij}\}_{i\in I}$ for every $j\in[m]$ still leaves a woven family of g-frames for $\hs$. One may ask whether or not a woven family of g-frames still is woven when some elements are removed from woven g-frames.

The following result gives condition on removing elements from woven g-frames and still be left with woven frames. We state this result for two g-frames.
\begin{proposition}
Suppose $\{\Lambda_i\}_{i\in I}$ and $\{\Gamma_i\}_{i\in I}$ are woven with universal constants $A$ and $B$. If $J\subset I$ and
$$\sum_{i \in J}\|\Lambda_if\|^2\le D\|f\|^2$$
for some $0<D<A$ and for all $f\in\hs$, then $\{\Lambda_i\}_{i\in I\setminus J}$ and $\{\Gamma_i\}_{i\in I\setminus J}$ are also g-frames for $\hs$ and are woven with universal lower and upper frame bounds $A-D$ and $B$, respectively.
\end{proposition}
\begin{proof}
The fact that $B$ is an upper weaving bound is obvious. Suppose that $\sigma\subset I\setminus J$. Then for all $f\in\hs$, we have
\begin{align*}
\sum_{i \in \sigma}\|\Lambda_if\|^2+\sum_{i \in (I\setminus J)\setminus\sigma}\|\Gamma_if\|^2&=\sum_{i \in \sigma\cup J}\|\Lambda_if\|^2-\sum_{i \in J}\|\Lambda_if\|^2+\sum_{i \in (I\setminus J)\setminus\sigma}\|\Gamma_if\|^2\\
&=\sum_{i \in \sigma\cup J}\|\Lambda_if\|^2+\sum_{i \in (\sigma\cup J)^c}\|\Gamma_if\|^2-\sum_{i \in J}\|\Lambda_if\|^2\\
&\ge (A-D)\|f\|^2.
\end{align*}
Thus a lower weaving bound is $A-D$. Taking $\sigma=J^c$ and $\sigma=\emptyset$ gives that $\{\Lambda_i\}_{i\in I\setminus J}$ and $\{\Gamma_i\}_{i\in I\setminus J}$ are g-frames for $\hs$, respectively.
\end{proof}
Since a g-frame is always woven with a copy of itself, we have the immediate corollary.
\begin{corollary}
If $\{\Lambda_i\}_{i\in I}$ is a g-frame with lower frame bound $A$ and 
$$\sum_{i \in J}\|\Lambda_i\|^2\le D\|f\|^2$$
for some $0<D<A$ and for all $f\in\hs$, then $\{\Lambda_i\}_{i\in J^c}$ is a g-frame with lower bound $A-D$.
\end{corollary}
We end this section by giving a relationship between the norms of the g-frame operators of the original g-frames and the weaving.
\begin{proposition}
Let $\{\Lambda_{ij}\}_{i\in I,j\in[m]}$ be a woven family of g-frames for $\hs$ with common frame bounds $A$ and $B$. Let $S_{\Lambda}^{(j)}$ be the frame operator of $\{\Lambda_{ij}\}_{i\in I}$ for each $j\in[m]$. For any partition $\sigma_j$ of $I$, if $S_{\Psi}$ represents the frame operator of $\Psi=\{\Lambda_{ij}\}_{i\in \sigma_j,j\in[m]}$, then for any $f\in\hs$,
$$\sum_{j\in [m]}\|(S_{\Lambda}^{(j)})_{\sigma_j}f\|^2\le B\|S_{\Psi}\|\|f\|^2,$$
where $(S_{\Lambda}^{(j)})_{\sigma_j}$ denotes the frame operator $S_{\Lambda}^{(j)}$ with sum restricted to $\sigma_j$.
\end{proposition}
\begin{proof}
Let $(T_{\Lambda}^{(j)})_{\sigma_j}$ be the synthesis operator of $\{\Lambda_{ij}\}_{i\in I}$ restricted to the sum over $\sigma_j$. Since $S_{\Lambda}^{(j)}\ge AI_{\hs}$, for any $f\in\hs$, we have
\begin{align*}
\sum_{j\in [m]}\|(S_{\Lambda}^{(j)})_{\sigma_j}f\|^2&=\sum_{j\in [m]}\|\sum_{i \in \sigma_j}\Lambda^*_{ij}\Lambda_{ij}f\|^2\\
&=\sum_{j\in [m]}\|(T_{\Lambda}^{(j)})_{\sigma_j}(T_{\Lambda}^{(j)})_{\sigma_j}^*f\|^2\\
&\le \sum_{j \in [m]}B\sum_{i \in \sigma_j}\|\Lambda_{ij}f\|^2\\
&=B\sum_{j \in [m]}\sum_{i \in \sigma_j}\|\Lambda_{ij}f\|^2 \\
&= B\left\langle S_{\Psi}f,f\right\rangle\\
&\le B\|S_{\Psi}\|\|f\|^2.
\end{align*}
This completes the proof of the proposition.
\end{proof}

\section{Weakly woven g-frames}
Proposition 2.2 shows one does not need to check for a universal upper frame bound because it is always given by the sum of the upper frame bounds. However, a universal lower bound is not clear in some cases such as in the infinite dimensional Hilbert spaces. To show that a universal bound must be obtained, we define a weaker form of weaving.
\begin{definition}
A family of g-frames $\{\Lambda_{ij}\}_{i\in \ns,j\in[m]}$ in $\hs$ is said to be weakly woven if for every partition $\sigma_j$ of $\ns$, the family $\{\Lambda_{ij}\}_{i\in \sigma_j,j\in[m]}$ is a g-frame for $\hs$.
\end{definition}
In the following proposition by extending Theorem 4.1 of \cite{bemrose2015weaving} , we give a characterization of weaving finite g-frames when the frame bounds are not clear.
\begin{theorem}
	A family of g-frames $\{\Lambda_{ij}\}_{i\in I,j\in[m]}$ for a finite-dimensional Hilbert space are woven if and only if for every partition $\{\sigma_j\}_{j\in [m]}$ of $I$, $\{\Lambda^*_{ij}(\hs_i)\}_{i\in\sigma_j,j\in[m]}$ spans the space.
\end{theorem}
\begin{proof}
The proof can be immediately obtained by Definition 3.1.
\end{proof}
The equivalent of woven and weakly woven frames is significantly more difficult to show for g-frames in an infinite dimensional. We need the following theorem in the finite dimensional case. The proof is based on the technique developed in Lemma 4.3 of \cite{bemrose2015weaving}. Further, one may observe that if $I=\ns$, then Lemma 4.3 of \cite{bemrose2015weaving} can be obtained from the following theorem.
\begin{theorem}
For each $j\in[m]$, let $\{\Lambda_{ij}\}_{i\in I}$ be a g-frame for $\hs$. Assume that a partition collection of disjoint finite sets $\{\tau_j\}_{j\in[m]}$ of $I$ and for every $A>0$ there exists a partition $\{\sigma_j\}_{j\in[m]}$ of the set $I\setminus(\tau_1\cup\cdots\cup\tau_m)$ such that $\{\Lambda_{ij}\}_{i\in(\sigma_j\cup\tau_j),j\in[m] }$ has a lower frame bound less than $A$. Then there exists a partition $\{\upsilon_j\}_{j\in [m]}$ of $I$ such that the $\{\Lambda_{ij}\}_{i\in\upsilon_j,j\in[m] }$ is not a g-frame for $\hs$. That is, these g-frames are not woven.
\end{theorem}
\begin{proof}
Since $I$ is a finite index, $I=\cup_{i\in\ns}I_i$, where $I_1,\cdots,I_i,\cdots$ are disjoint index sets. Suppose $\tau_{ij}=\emptyset$ for all $j\in [m]$ and $A=1$. Then there exists a partition $\{\sigma_{1j}\}_{j\in [m]}$ of $I$ such that $\{\Lambda_{ij}\}_{i\in (\sigma_{1j}\cup\tau_{1j}),j\in[m] }$ has a lower bound less than 1. Therefore, there exists a vector $f_1\in\hs$ with $\|f_1\|=1$ such that 
$$\sum_{j \in [m]}\sum_{i\in (\sigma_{1j}\cup\tau_{1j})}\|\Lambda_{ij}f_1\|^2=\sum_{i\in (\sigma_{11}\cup\tau_{11})}\|\Lambda_{i1}f_1\|^2+\cdots+\sum_{i\in (\sigma_{1m}\cup\tau_{1m})}\|\Lambda_{im}f_1\|^2<1.$$
Since 
$$\sum_{j \in [m]}\sum_{i \in I}\|\Lambda_{ij}f_1\|^2=\sum_{i \in I}\|\Lambda_{i1}f_1\|^2+\cdots+\sum_{i \in I}\|\Lambda_{im}f_1\|^2<\infty,$$
there is a $k_1\in\ns$ so that
$$\sum_{j \in [m]}\sum_{i \in \kappa_1}\|\Lambda_{ij}f_1\|^2<1,$$
where $\kappa_1=\cup_{i\ge k_1+1}I_i$. Choose a partition $\{\tau_{2j}\}_{j\in[m]}$ of $I_1\cup\cdots \cup I_{k_1}$ such that $\tau_{2j}=\tau_{1j}\cup(\sigma_{1j}\cap(I_1\cup\cdots \cup I_{k_1}))$ for all $j\in[m]$ and $A=\frac{1}{2}$. Then there exists a partition $\{\sigma_{2j}\}_{j\in[m]}$ of $I\setminus(I_1\cup\cdots\cup  I_{k_1})$ such that $\{\Lambda_{ij}\}_{i\in (\sigma_{2j}\cup\tau_{2j}),j\in[m] }$ has a lower bound less than $\frac{1}{2}$. Therefore, there exists a vector $f_2\in\hs$ with $\|f_2\|=1$ such that
$$\sum_{j \in [m]}\sum_{i\in (\sigma_{2j}\cup\tau_{2j})}\|\Lambda_{ij}f_2\|^2=\sum_{i\in (\sigma_{21}\cup\tau_{21})}\|\Lambda_{i1}f_2\|^2+\cdots+\sum_{i\in (\sigma_{2m}\cup\tau_{2m})}\|\Lambda_{im}f_2\|^2<\frac{1}{2}.$$
Again, since
$$\sum_{j \in [m]}\sum_{i \in I}\|\Lambda_{ij}f_2\|^2<\infty,$$
we can find an integer $k_2>k_1$ such that 
$$\sum_{j \in [m]}\sum_{i \in \kappa_2}\|\Lambda_{ij}f_2\|^2<\frac{1}{2},$$
where $\kappa_2=\cup_{i\ge k_2+1}I_i$.\\
Continuing in this way, for $A=\frac{1}{p}$ and for a partition $\{\tau_{pj}\}_{j\in[m]}$ of $I_1\cup\cdots \cup I_{k_{p-1}}$ such that $$\tau_{pj}=\tau_{(p-1)j}\cup(\sigma_{(p-1)j}\cap(I_1\cup\cdots \cup I_{k_{p-1}}))$$
for all $j\in[m]$, there exists a partition $\{\sigma_{pj}\}_{j\in[m]}$ of $I\setminus(I_1\cup\cdots \cup I_{k_{p-1}})$ such that $\{\Lambda_{ij}\}_{i\in (\sigma_{pj}\cup\tau_{pj}),j\in[m] }$ has a lower bound less than $\frac{1}{p}$. Thus, there exists a $f_p\in\hs$ with $\|f_p\|=1$ such that 
$$\sum_{j \in [m]}\sum_{i\in (\sigma_{pj}\cup\tau_{pj})}\|\Lambda_{ij}f_p\|^2=\sum_{i\in (\sigma_{p1}\cup\tau_{p1})}\|\Lambda_{i1}f_p\|^2+\cdots+\sum_{i\in (\sigma_{pm}\cup\tau_{pm})}\|\Lambda_{im}f_p\|^2<\frac{1}{p}.$$
Now 
$$\sum_{j \in [m]}\sum_{i \in I}\|\Lambda_{ij}f_p\|^2<\infty,$$
there exists a $k_p>k_{p-1}$ such that 
$$\sum_{j \in [m]}\sum_{i \in \kappa_p}\|\Lambda_{ij}f_p\|^2<\frac{1}{p},$$
where $\kappa_p=\cup_{i\ge k_p+1}I_i$. Choose a partition $\{\upsilon_j\}_{j\in[m]}$ of $I$, where $\upsilon_j=\cup_{i\in\ns}\{\tau_{ij}\}=\tau_{p+1j}\cup(\upsilon_j\cap I\setminus(I_1\cup\cdots\cup  I_{{p}}) )$. Then the family  $\{\Lambda_{ij}\}_{i\in\upsilon_j,j\in[m] }$ is not a g-frame for $\hs$. If not, let $\{\Lambda_{ij}\}_{i\in\upsilon_j,j\in[m] }$ be a g-frame for $\hs$ with frame bounds $C$ and $D$, respectively. Thus, by using the Archimedean Property, there exists a $q\in\ns$ such that $q>\frac{2}{C}$. There exists a $f_q\in\hs$ with $\|f_q\|=1$ such that  
\begin{align*}
	\sum_{j\in [m]}\sum_{i\in\upsilon_j}\|\Lambda_{ij}f_q\|^2&=\sum_{i\in\upsilon_1}\|\Lambda_{i1}f_q\|^2+\sum_{i\in\upsilon_2}\|\Lambda_{i2}f_q\|^2+\cdots+\sum_{i\in\upsilon_m}\|\Lambda_{im}f_q\|^2\\
	&=\bigg[ \sum_{i\in\tau_{(q+1)1}}\|\Lambda_{i1}f_q\|^2+\cdots+\sum_{i\in\tau_{(q+1)m}}\|\Lambda_{im}f_q\|^2\bigg] \\ &\quad+\bigg[ \sum_{i\in\upsilon_1\cap I\setminus(I_1\cup\cdots I_{{q}})}\|\Lambda_{i1}f_q\|^2+\cdots+\sum_{i\in\upsilon_m\cap I\setminus(I_1\cup\cdots I_{{q}})}\|\Lambda_{im}f_q\|^2\bigg] \\
	&\le \bigg[ \sum_{i\in(\sigma_{q1}\cup\tau_{q1})}\|\Lambda_{i1}f_q\|^2+\cdots+\sum_{i\in(\sigma_{qm}\cup\tau_{qm})}\|\Lambda_{im}f_q\|^2\bigg] \\ 
	&\quad+\bigg[ \sum_{i\in\cup_{i\ge q+1}I_i}\|\Lambda_{i1}f_q\|^2+\cdots+\sum_{i\in\cup_{i\ge q+1}I_i}\|\Lambda_{im}f_q\|^2\bigg] \\
	&< \frac{1}{q}+\frac{1}{q}=\frac{2}{q}\|f_q\|^2<C\|f_q\|^2,
\end{align*}
which is a contradiction because $C$ is the lower bound of $\{\Lambda_{ij}\}_{i\in\upsilon_j,j\in[m] }$. This completes the proof.
\end{proof}
The outcome in Theorem 3.3 gives a necessary condition for weaving g-frames.
\begin{corollary}
Let $\{\Lambda_{ij}\}_{i\in I,j\in[m]}$ be a family of woven g-frames for $\hs$. Then there exists a collection of disjoint finite subsets $\{\tau_j\}_{j\in[m]}$ of $I$ and $A>0$ such that for any partition $\{\sigma_j\}_{j\in I}$ of the set $I\setminus(\tau_1\cup\cdots\cup\tau_m)$, the family $\{\Lambda_{ij}\}_{i\in(\sigma_j\cup\tau_j),j\in[m] }$ is a g-frame for $\hs$ with lower frame bound $A$.
\end{corollary}
Next, we give a sufficient condition for weaving g-frames.
\begin{theorem}
For each $j\in[m]$, let $\{\Lambda_{ij}\}_{i\in I}$ be a g-frame for $\hs$ with bounds $A_j$ and $B_j$. Suppose there exists $K>0$ such that
\begin{equation*}
\sum_{i \in  J}\|(\Lambda_{ij}-\Lambda_{il})f\|^2\le K\min\{\sum_{i \in  J}\|\Lambda_{ij}f\|^2,\sum_{i \in  J}\|\Lambda_{il}f\|^2\}~~(j,l\in[m],~j\ne l)
\end{equation*}
for all $f\in\hs$ and for all subsets $J\subset I$. Then the family of g-frames $\{\{\Lambda_{ij}\}_{i\in I}:j\in[m]\}$ is woven with universal frame bounds $\frac{\sum_{j\in [m]}A_j}{2(m-1)(K+1)+1}$ and $\sum_{j \in [m]}B_j$.
\end{theorem}
\begin{proof}
	From Proposition 2.2, we know that the $\{\{\Lambda_{ij}\}_{i\in I}:j\in[m]\}$ satisfies upper frame inequality with universal upper frame bound $\sum_{j \in [m]}B_j$. Let $\{\sigma_j\}_{j\in[m]}$ be any partition of $I$. For the lower frame inequality, we have
\begin{align*}
	\sum_{j\in [m]}A_j\|f\|^2&=A_1\|f\|^2+\cdots+A_m\|f\|^2\\
	&\le \sum_{i \in  I}\|\Lambda_{i1}f\|^2+\cdots+\sum_{i \in  I}\|\Lambda_{im}f\|^2\\ 
	&=\bigg(\sum_{i \in  \sigma_1}\|\Lambda_{i1}f\|^2+\cdots+\sum_{i \in  \sigma_m}\|\Lambda_{i1}f\|^2\bigg)+\cdots\\
	&\quad+\bigg(\sum_{i \in  \sigma_1}\|\Lambda_{im}f\|^2+\cdots+\sum_{i \in  \sigma_m}\|\Lambda_{im}f\|^2\bigg)\\
	&\le  \bigg[ \sum_{i \in  \sigma_1}\|\Lambda_{i1}f\|^2+2\bigg( \sum_{i \in  \sigma_2}\|(\Lambda_{i1}-\Lambda_{i2})f\|^2+\sum_{i \in  \sigma_2}\|\Lambda_{i2}f\|^2\bigg)+\cdots\\ 
	&\qquad+2\bigg( \sum_{i \in  \sigma_m}\|(\Lambda_{i1}-\Lambda_{im})f\|^2+\sum_{i \in  \sigma_m}\|\Lambda_{im}f\|^2\bigg) \bigg]+\cdots\\
	&\quad+\bigg[2\bigg( \sum_{i \in  \sigma_1}\|(\Lambda_{im}-\Lambda_{i1})f\|^2+\sum_{i \in  \sigma_1}\|\Lambda_{i1}f\|^2\bigg) +\cdots\\
	&\qquad+2\bigg( \sum_{i \in  \sigma_{m-1}}\|(\Lambda_{im}-\Lambda_{i(m-1)})f\|^2+\sum_{i \in  \sigma_{m-1}}\|\Lambda_{i(m-1)}f\|^2\bigg) \\
	&\qquad+\sum_{i \in  \sigma_{m}}\|\Lambda_{im}f\|^2\bigg]\\
	&\le \bigg[ \sum_{i \in  \sigma_1}\|\Lambda_{i1}f\|^2+2\bigg( K\sum_{i \in  \sigma_2}\|\Lambda_{i2}f\|^2+\sum_{i \in  \sigma_2}\|\Lambda_{i2}f\|^2\bigg)+\cdots\\ 
	&\qquad+2\bigg( K\sum_{i \in  \sigma_m}\|\Lambda_{im}f\|^2+\sum_{i \in  \sigma_m}\|\Lambda_{im}f\|^2\bigg) \bigg]+\cdots\\
	&\quad+\bigg[2\bigg( K\sum_{i \in  \sigma_1}\|\Lambda_{i1}f\|^2+\sum_{i \in  \sigma_1}\|\Lambda_{i1}f\|^2\bigg) +\cdots\\
	&\qquad+2\bigg( K\sum_{i \in  \sigma_{m-1}}\|\Lambda_{i(m-1)}f\|^2+\sum_{i \in  \sigma_{m-1}}\|\Lambda_{i(m-1)}f\|^2\bigg) \\
	&\qquad+\sum_{i \in  \sigma_{m}}\|\Lambda_{im}f\|^2\bigg]\\
	&=\sum_{i \in  \sigma_1}\|\Lambda_{i1}f\|^2+\cdots+\sum_{i \in  \sigma_{m}}\|\Lambda_{im}f\|^2\\
	&\quad+(m-1)2(K+1)\bigg(\sum_{i \in  \sigma_1}\|\Lambda_{i1}f\|^2+\cdots+\sum_{i \in  \sigma_{m}}\|\Lambda_{im}f\|^2\bigg)\\
	&=[2(m-1)(K+1)+1]\sum_{j \in [m]}\sum_{i \in  \sigma_j}\|\Lambda_{ij}f\|^2
\end{align*}	
for all $f\in\hs$. Hence, for all $f\in\hs$, we have
$$\frac{\sum_{j\in [m]}A_j}{2(m-1)(K+1)+1}\|f\|^2\le \sum_{j \in [m]}\sum_{i \in  \sigma_j}\|\Lambda_{ij}f\|^2\le \sum_{j \in [m]}B_j\|f\|^2.$$
The proof is completed.
\end{proof}
The authors of \cite{bemrose2015weaving} showed that the weakly woven is equivalent to the frames being woven.
\begin{theorem}\cite{bemrose2015weaving}
Given two frames $\{f_i\}_{i=1}^{\infty}$ and $\{g_i\}_{i=1}^{\infty}$ for $\hs$, the following are equivalent:
\begin{enumerate}
	\item The two frames are woven.
	
	\item The two frames are weakly woven.
\end{enumerate}
\end{theorem}
In fact, the above result is also satisfied in the case of g-frames.
\begin{theorem}
Two g-frames $\{\Lambda_i\}_{i\in \ns}$ and $\{\Gamma_i\}_{i\in\ns}$ for $\hs$ are woven if and only if they are weakly woven.
\end{theorem}
\begin{proof}
The proof can be easily obtained by Lemma 1.3 and Theorem 3.6.
\end{proof}
\section{Weaving g-Riesz Bases}
In this section we classify when g-Riesz bases and g-Riesz basis sequences can be woven and consider the weaving equivalent of an unconditional g-basis for $\hs$.

The following result is an extension of Theorem 5.2 of \cite{bemrose2015weaving} to g-Riesz bases.
\begin{theorem}
Let $\{\Lambda_{i}\}_{i\in \ns}$ and $\{\Gamma_{i}\}_{i\in \ns}$ be two g-Riesz bases for which there are common constants $0<A\le B<\infty$ so that for every $\sigma$ of $\ns$, the family $\{\Lambda_{i}\}_{i\in \sigma}\cup\{\Gamma_{i}\}_{i\in \sigma^c}$ is a g-Riesz sequence with Riesz bounds $A$ and $B$. Then for every partition $\sigma\subset \ns$ the  $\{\Lambda_{i}\}_{i\in \sigma}\cup\{\Gamma_{i}\}_{i\in \sigma^c}$ is actually a g-Riesz basis, that is, the two g-Riesz bases are woven.
\end{theorem}
\begin{proof}
We proceed by induction on the cardinality of $\sigma$. \\
First, we assume that $|\sigma|<\infty$.  The case $|\sigma|=0$ being obvious, we assume the result holds for every $\sigma$ with $|\sigma|=k$.

Now let $\sigma\subset \ns$ with $|\sigma|=k+1$ and choose $i_0\in\sigma$. Let $\sigma_1=\sigma\setminus\{i_0\}$, then $\{\Lambda_{i}\}_{i\in \sigma_1}\cup\{\Gamma_{i}\}_{i\in \sigma_1^c}$ is a g-Riesz basis by the induction hypothesis.

We proceed by way if contradiction assume that $\{\Lambda_{i}\}_{i\in \sigma}\cup\{\Gamma_{i}\}_{i\in \sigma^c}$ is not a g-Riesz basis. However, it is at least a g-Riesz sequence by assumption. For any $g\in\hs_{i_0}$, if
$$\Gamma_{i_0}^*(g)\in {\rm span}\big(\{\Lambda^*_{i}(\hs_i)\}_{i\in \sigma}\cup\{\Gamma^*_{i}(\hs_i)\}_{i\in \sigma^c}\big),$$ 
then 
$${\rm\overline{span}}\big(\{\Lambda^*_{i}(\hs_i)\}_{i\in \sigma}\cup\{\Gamma^*_{i}(\hs_i)\}_{i\in \sigma^c}\big)\supset{\rm\overline{span}}\big(\{\Lambda^*_{i}(\hs_i)\}_{i\in \sigma_1}\cup\{\Gamma^*_{i}(\hs_i)\}_{i\in \sigma_1^c}\big)=\hs,$$
i.e., $\{\Lambda_{i}\}_{i\in \sigma}\cup\{\Gamma_{i}\}_{i\in \sigma^c}$ would be a g-basis, which is assumed to not be the case. So it must be that 
$$\Gamma_{i_0}^*(g)\notin {\rm span}\big(\{\Lambda^*_{i}(\hs_i)\}_{i\in \sigma}\cup\{\Gamma^*_{i}(\hs_i)\}_{i\in \sigma^c}\big)$$ 
from which it follows that
$$\{\Lambda_{i}\}_{i\in \sigma}\cup\{\Gamma_{i}\}_{i\in \sigma^c}\cup\{\Gamma_{i_0}\}$$
is a g-Riesz sequence in $\hs$. Hence, because $\sigma_1^c=\sigma^c\cup\{i_0\}$, 
$$\{\Lambda_{i}\}_{i\in \sigma_1}\cup\{\Gamma_{i}\}_{i\in \sigma_1^c}$$
cannot be a g-Riesz basis, since we obtained it by deleting the element $\Gamma_{i_0}$ from a g-Riesz sequence, which leads to a contradiction. 

Next, by way of contradiction, assume there is a $\sigma\subset\ns$ with both $\sigma$ and $\sigma^c$ infinite, so that
$$\widetilde{\hs}={\rm\overline{span}}\big(\{\Lambda^*_{i}(\hs_i)\}_{i\in \sigma}\cup\{\Gamma^*_{i}(\hs_i)\}_{i\in \sigma^c}\big)\ne \hs.$$
Choose a nonzero $f\in\widetilde{\hs}^{\perp}$. Since $\{\Gamma_i\}_{i\in I}$ is g-Bessel sequence, by taking the tail the series, there exists a $\sigma_1\subset\sigma$ with $|\sigma_1|<\infty$ and 
$$\sum_{i\in \sigma\setminus\sigma_1}\|\Gamma_{i}f\|^2<\frac{A}{2}\|f\|^2.$$ From the first part of proof, the family 
$$\{\Lambda_{i}\}_{i\in\sigma_1}\cup\{\Gamma_i\}_{i\in\sigma\setminus\sigma_1}\cup\{\Gamma_i\}_{i\in\sigma^c}$$
is a g-Riesz basis with bounds $A$, $B$ and therefore 
\begin{align*}
	A\|f\|^2&\le \sum_{i \in  \sigma_1}\|\Lambda_if\|^2+\sum_{i \in  \sigma\setminus\sigma_1}\|\Gamma_if\|^2+\sum_{i \in  \sigma^c}\|\Lambda_if\|^2\\
	&< \frac{A}{2}\|f\|^2,
\end{align*}
giving a contradiction.
\end{proof}
By extending the Theorem 5.3 of \cite{bemrose2015weaving}, we show that if two g-Riesz bases are woven, then every weaving is in fact a g-Riesz basis, and not just a g-frame.
\begin{theorem}
Suppose $\{\Lambda_{i}\}_{i\in \ns}$ and $\{\Gamma_{i}\}_{i\in \ns}$ are g-Riesz bases and that there is a common constant $A>0$ so that for every $\sigma\subset \ns$, the family $\{\Lambda_{i}\}_{i\in \sigma}\cup\{\Gamma_{i}\}_{i\in \sigma^c}$ is a g-frame with lower frame bound $A$. Then for every $\sigma\subset \ns$, the family $\{\Lambda_{i}\}_{i\in \sigma}\cup\{\Gamma_{i}\}_{i\in \sigma^c}$ is actually a g-Riesz basis.
\end{theorem}
\begin{proof}
First, assume $|\sigma|<\infty$, we do the proof by induction on $|\sigma|$ with $|\sigma|=0$ clear. Now we assume the result holds for every $\sigma$ with $|\sigma|=n$. Let $\sigma\subset\ns $ be so that $|\sigma|=n+1$ and let $i_0\in \sigma$. Then $\{\Lambda_{i}\}_{i\in \sigma\setminus\{i_0\}}\cup\{\Gamma_{i}\}_{i\in \sigma^c\cup\{i_0\}}$ is a g-Riesz basis and therefore
$$\{\Lambda_{i}\}_{i\in \sigma\setminus\{i_0\}}\cup\{\Gamma_{i}\}_{i\in \sigma^c}$$
is a g-Riesz sequence spanning a subspace of co-dimension at least one.\\
Now, by assumption, $\{\Lambda_{i}\}_{i\in \sigma}\cup\{\Gamma_{i}\}_{i\in \sigma^c}$ is at least a g-frame. Since the removal of the single vector $\Gamma_{i_0}$ yields a set that does not longer span $\hs$, $\{\Lambda_{i}\}_{i\in \sigma}\cup\{\Gamma_{i}\}_{i\in \sigma^c}$ must actually be a g-Riesz basis \cite{zhu2008}. Furthermore, its lower bound is $A$.

Next, Let $|\sigma|=\infty$. By choosing $\sigma_1\subset\sigma_2\subset\cdots\subset\sigma$ such that
$$\sigma=\bigcup_{j=1}^{\infty}\sigma_j,$$
and $|\sigma_j|<\infty$. Now, for every $j=1,2,\cdots$ the family
$$\{\Lambda_{i}\}_{i\in \sigma_j}\cup\{\Gamma_{i}\}_{i\in \sigma\setminus\sigma_j}\cup\{\Gamma_{i}\}_{i\in\sigma^c}=\{\Lambda_{i}\}_{i\in \sigma_j}\cup\{\Gamma_{i}\}_{i\in\sigma_j^c}$$
is a g-Riesz basis with lower bound $A$. If $\{g_i\}_{i=1}^{\infty}\in \oplus_{i\in I}\hs_i$ and
$$\sum_{i\in\sigma}\Lambda^*g_i+\sum_{i\in\sigma^c}\Gamma^*g_i=0,$$
then
\begin{align*}
	0&=\|\sum_{i\in\sigma}\Lambda^*g_i+\sum_{i\in\sigma^c}\Gamma^*g_i\|^2=\lim_{j\rightarrow\infty}\|\sum_{i\in\sigma}\Lambda^*g_i+\sum_{i\in\sigma^c}\Gamma^*g_i\|^2\\
	&\ge \lim_{j\rightarrow\infty}A\big(\sum_{i \in  \sigma_j}|a_i|^2+\sum_{i \in  \sigma^c_j}|a_i|^2\big)
\end{align*}
where the last inequality follows from the g-Riesz basis property of $\{\Lambda_{i}\}_{i\in \sigma_j}\cup\{\Gamma_{i}\}_{i\in \sigma_j^c}$. So $g_i=0$ for every $i=1,2,\cdots,$ implying that the synthesis operator for the family $\{\Lambda_{i}\}_{i\in \sigma}\cup\{\Gamma_{i}\}_{i\in \sigma^c}$ is bounded, linear, onto, and by the above it is also one-to-one. Therefore, it is invertible and so the family $\{\Lambda_{i}\}_{i\in \sigma}\cup\{\Gamma_{i}\}_{i\in \sigma^c}$ is a g-Riesz basis.
\end{proof}
The following result by extending Theorem 5.4 of \cite{bemrose2015weaving} says that a g-frame (which is not a g-Riesz basis) cannot be woven with a g-Riesz basis.
\begin{theorem}
Let $\Lambda=\{\Lambda_i\}_{i\in\ns}$ be a g-Riesz basis and let $\Gamma=\{\Gamma_i\}_{i\in\ns}$ be a g-frame for $\hs$. If $\Lambda$ and $\Gamma$ are woven, then $\Gamma$ must actually be a g-Riesz basis.
\end{theorem}
\begin{proof}
Without loss of generality, assume that $\Lambda$ is an g-orthonormal basis. By way of contradiction, assume that $\Gamma$ is not a g-Riesz basis. It may be assumed that $\Gamma^*_1(g)\in {\rm\overline{span}}\{\Gamma^*_i(\hs_i)\}_{i\ne 1,i\in\ns}$ for any $g\in\hs_1$. Now, choose $n\in\ns$ such that
$$0\le d(\Gamma^*_1(g),{\rm\overline{span}}\{\Gamma^*_i(\hs_i)\}_{i=2}^n)\le \varepsilon$$
and let 
$$\widetilde{\hs}_n=\big[{\rm\overline{span}}\{\Gamma^*_i(\hs_i)\}_{i=2}^n\big]^{\perp}.$$
Then $\widetilde{\hs}_n$ has co-dimension at most $n-1$ in $\hs$ and since $\Lambda$ is an g-orthonormal basis,
$$\dim {\rm{span}}\{\Lambda^*_i(\hs_i)\}_{i=1}^n=n.$$
So there exists  $f\in {\rm{span}}\{\Lambda^*_i(\hs_i)\}_{i=1}^n\cap\widetilde{\hs}_n$ with $\|f\|=1$. Now, if $\sigma^c=[n]$ then 
$$\sum_{i \in  \sigma}\|\Lambda_{i}f\|^2=0,$$
while
$$\sum_{i \in  \sigma^c}\|\Gamma_{i}f\|^2=\|\Gamma_{1}f\|^2\le \varepsilon.$$
So these two families are not woven.
\end{proof}
The next result gives a necessary and sufficient condition such that a g-frame and a nonidentical recoding of itself can be woven.
\begin{proposition}
If $\{\Lambda_{i}\}_{i\in I}$ is a g-Riesz basis with bounds $A, B$ and $\pi$ is a permutation of $I$, then for every $\sigma\subset I$ the family $\{\Lambda_{i}\}_{i\in\sigma}\cup\{\Lambda_{\pi{(i)}}\}_{i\in\sigma^c}$ is a g-frame sequence with bounds $A$ and $2B$. Moreover, $\{\Lambda_i\}_{i\in I}$ and $\{\Lambda_{\pi(i)}\}_{i\in I}$ are woven if and only if $\pi(i)=i$ for all $i\in I$.
\end{proposition}
\begin{proof}
For any $f\in{\rm \overline{span}}(\{\Lambda_i^*(\hs_i)\}_{i\in \sigma}\cup\{\Lambda_{\pi{(i)}}^*(\hs_{\pi{(i)}})\}_{i\in \sigma^c})$ and for any $\sigma\subset I$, we have
$$\sum_{i \in  \sigma}\|\Lambda_if\|^2+\sum_{i \in  \sigma^c}\|\Lambda_{\pi{(i)}}f\|^2\ge \sum_{i \in  I,\pi{(i)}\notin\sigma^c}\|\Lambda_if\|^2\ge A\|f\|^2, $$
since any g-subsequence of a g-Riesz basis is a g-Riesz sequence with the same bounds. The upper frame bound is the sum of the upper frames bounds, which is $2B$. Note that it is not $B$ due to redundancy.

The moreover part is now proven by contradiction. Assume $\pi(i)\ne i$ so that $\pi(i_0)=j_0\ne i_0$ for some $i_0,j_0\in I$. Let $\sigma=I\setminus\{i_0\}$. Then 
$$\{\Lambda_{i}\}_{i\in\sigma}\cup\{\Lambda_{\pi{(i)}}\}_{i\in\sigma^c}=\{\Lambda_{i}\}_{i\in I\setminus\{i_0\}}\cup\{\Lambda_{j_0}\}$$
which is the set in which $\Lambda_{j_0}$ appears twice, but $\Lambda_{i_0}$ does not appear at all and therefore the closure of the span is not the whole space.
\end{proof}
We now give the weaving equivalent of an unconditional g-basis for $\hs$.
\begin{theorem}
Let $\{\Lambda_i\}_{i\in\ns}$ and $\{\Gamma_i\}_{i\in\ns}$ be g-Riesz basis sequences for $\hs$ with bounds $A_1, B_1$ and $A_2, B_2$ respectively. Then the following are equivalent:
\begin{enumerate}
	\item There exist constants $0<B\le C<\infty$ so that for every $\sigma\in\ns$the family $\{\Lambda_{i}\}_{i\in \sigma}\cup\{\Gamma_{i}\}_{i\in \sigma^c}$ is a g-Riesz basis sequence with bounds $B, C$.
	\item There is a constant $A>0$ satisfying for all $\{g_i\}_{i\in\ns}\in\oplus_{i\in\ns}\hs_i$ and all $\sigma\in\ns$
	$$A\|\sum_{i \in  \sigma}\Lambda_i^*g_i\|^2\le \|\sum_{i \in  \sigma}\Lambda_i^*g_i+\sum_{i \in  \sigma^c}\Gamma_i^*g_i\|^2.$$
	\item There is a constant $D>0$ satisfying for all $\{g_i\}_{i\in\ns}\in\oplus_{i\in\ns}\hs_i$ and all $\sigma\in\ns$
	$$D\big(\|\sum_{i \in  \sigma}\Lambda_i^*g_i\|^2+\|\sum_{i \in  \sigma^c}\Gamma_i^*g_i\|^2\big)\le \|\sum_{i \in  \sigma}\Lambda_i^*g_i+\sum_{i \in  \sigma^c}\Gamma_i^*g_i\|^2.$$
	\item There is a constant $E>0$ satisfying for all $\{g_i\}_{i\in\ns}\in\oplus_{i\in\ns}\hs_i$ and all $\sigma\in\ns$ so that if $\|\sum_{i \in  \sigma}\Lambda_i^*g_i\|=1$, then
	$$E\le \|\sum_{i \in  \sigma}\Lambda_i^*g_i+\sum_{i \in  \sigma^c}\Gamma_i^*g_i\|^2.$$
\end{enumerate}
\end{theorem} 
\begin{proof}
The implications (3)$\Rightarrow$(2), and (2)$\Rightarrow$(4) are clear. \\
We now prove (2)$\Rightarrow$(3). Given the assumptions in (2) we compute
\begin{align*}
	\|\sum_{i \in  \sigma^c}\Gamma_i^*g_i\|^2&=\|\sum_{i\in\sigma^c}\Gamma_i^*g_i+\sum_{i\in\sigma}\Lambda_i^*g_i-\sum_{i\in\sigma}\Lambda_i^*g_i\|^2\\
	&\le\|\sum_{i\in\sigma^c}\Gamma_i^*g_i+\sum_{i\in\sigma}\Lambda_i^*g_i\|^2+\|\sum_{i\in\sigma}\Lambda_i^*g_i\|^2 \\
	&\le \|\sum_{i\in\sigma^c}\Gamma_i^*g_i+\sum_{i\in\sigma}\Lambda_i^*g_i\|^2+\frac{1}{A}\|\sum_{i\in\sigma^c}\Gamma_i^*g_i+\sum_{i\in\sigma}\Lambda_i^*g_i\|^2.
\end{align*}

Hence, 
$$\frac{A}{A+1}\|\sum_{i \in  \sigma^c}\Gamma_i^*g_i\|^2\le \|\sum_{i\in\sigma^c}\Gamma_i^*g_i+\sum_{i\in\sigma}\Lambda_i^*g_i\|^2.$$
Similarly,
$$\frac{A}{A+1}\|\sum_{i \in  \sigma}\Lambda_i^*g_i\|^2\le \|\sum_{i\in\sigma^c}\Gamma_i^*g_i+\sum_{i\in\sigma}\Lambda_i^*g_i\|^2.$$
Therefore we have
\begin{align*}
	\frac{1}{2}\frac{A}{A+1}\big(\|\sum_{i \in  \sigma^c}\Gamma_i^*g_i\|^2+\|\sum_{i\in\sigma}\Lambda^*g_i\|^2\big)&\le\frac{A}{A+1}\max\big(\|\sum_{i \in  \sigma^c}\Gamma_i^*g_i\|^2,\|\sum_{i\in\sigma}\Lambda_i^*g_i\|^2\big)\\
	&\le\|\sum_{i\in\sigma^c}\Gamma_i^*g_i+\sum_{i\in\sigma}\Lambda_i^*g_i\|^2.
\end{align*}
(4)$\Rightarrow$(2): If $\sum_{i \in  \sigma}\Lambda_i^*g_i=0$ we are done. So assume not and by (4) we have
$$E\le \frac{1}{\|\sum_{i \in  \sigma}\Lambda_i^*g_i\|^2}\|\sum_{i\in\sigma^c}\Gamma_i^*g_i+\sum_{i\in\sigma}\Lambda_i^*g_i\|^2.$$
So (2) holds.

At this point we know that (2)$\Leftrightarrow$(3)$\Leftrightarrow$(4). Now we prove (1)$\Rightarrow$(2). Given $\sigma$ and $\{g_i\}_{i\in\ns}\in\oplus_{i\in\ns}\hs_i$, by (1) we have
\begin{align*}
	\|\sum_{i \in  \sigma}\Lambda_i^*g_i\|^2&\le C\sum_{i \in \sigma}\|g_i\|^2\le C\big(\sum_{i \in \sigma}\|g_i\|^2+\sum_{i \in \sigma^c}\|g_i\|^2\big)\\
	&\le\frac{C}{B}\|\sum_{i\in\sigma^c}\Gamma_i^*g_i+\sum_{i\in\sigma}\Lambda_i^*g_i\|^2.
\end{align*}
Finally we prove (3)$\Rightarrow$(1). For all $\{g_i\}_{i\in\ns}\in\oplus_{i\in\ns}\hs_i$ and $\sigma\in\ns$ we have
\begin{align*}
    \sum_{i \in\ns}\|g_i\|^2&=\sum_{i \in\sigma}\|g_i\|^2+\sum_{i \in\sigma^c}\|g_i\|^2\le \frac{1}{A_1}\|\sum_{i \in  \sigma}\Lambda_i^*g_i\|^2+\frac{1}{A_2}\|\sum_{i \in  \sigma^c}\Gamma_i^*g_i\|^2\\
	&\le\max\{\frac{1}{A_1},\frac{1}{A_2}\}\big( \|\sum_{i \in  \sigma}\Lambda_i^*g_i\|^2+\|\sum_{i \in  \sigma^c}\Gamma_i^*g_i\|^2\big)\\
	&\le \frac{1}{D}\max\{\frac{1}{A_1},\frac{1}{A_2}\}\|\sum_{i\in\sigma^c}\Gamma_i^*g_i+\sum_{i\in\sigma}\Lambda_i^*g_i\|^2
\end{align*}
Hence the lower bound is obtained. The upper of $\{\Lambda_{i}\}_{i\in \sigma}\cup\{\Gamma_{i}\}_{i\in \sigma^c}$ is obvious. The proof of the theorem is completed.
\end{proof}
\section{Perturbation theorem for weaving g-frames}
In this section, We present Paley-Wiener-type perturbation results \cite{christensen1995paley} for weaving g-frames. It is shown that the family of g-frames is woven under small perturbation. Specifically, we have the following. 

\begin{theorem}
For each $j\in[m]$, let $\Lambda_j=\{\Lambda_{ij}\}_{i\in I}$ be a g-frame for $\hs$ with frame bounds $A_j$ and $B_j$. Assume that there exist non-negative scalars $\lambda_j,~\eta_j,~\mu_j,~(j\in[m])$ such that for some fixed $n\in[m]$,
$$A=A_n-\sum_{j\in[m]\setminus\{n\}}(\lambda_j+\eta_j\sqrt{B_n}+\mu_j\sqrt{B_j})(\sqrt{B_n}+\sqrt{B_j})>0$$
and 
$$\|\sum_{i\in J}(\Lambda_{in}^*-\Lambda_{ij}^*)g_i\|\le \eta_j\|\sum_{i\in J}\Lambda_{in}^*g_i\|+\mu_j\|\sum_{i\in J}\Lambda_{ij}^*g_i\|+\lambda_j(\sum_{i\in J}\|g_i\|^2)^{1/2}$$
for any finite subset $J\subset I$, $g_i\in\hs_i$ and $j\in[m]\setminus\{n\}$. Then for any partition $\{\sigma_j\}_{j\in[m]}$ of $I$, the family $\{\Lambda_{ij}\}_{i\in \sigma_j,j\in[m]}$ is a g-frame for $\hs$ with universal frame bounds $A$ and $\sum_{j\in[m]}B_j$. Hence the family of g-frames $\{\Lambda_j\}_{j\in[m]}$ for $\hs$ is woven.
\end{theorem}
\begin{proof}
By Proposition 2.2, for any partition $\{\sigma_j\}_{j\in[m]}$ of $I$, the family $\{\Lambda_{ij}\}_{i\in \sigma_j,j\in[m]}$ is a g-Bessel sequence with Bessel bound $\sum_{j\in[m]}B_j$.

For the lower frame inequality, let $T_{\Lambda}^{(i)}$ be a synthesis operator associated with the g-frame $\{\Lambda_{ij}\}_{i\in I}$ for $j\in[m]$. Since 
\begin{align*}
	\|T_{\Lambda}^{(j)}g_i\|&=\|\sum_{i\in J}\Lambda_{ij}^*g_i\|=\sup_{\|g\|=1}\left| \left\langle g,\sum_{i\in J}\Lambda_{ij}^*g_i\right\rangle \right|\\
	&\le\sup_{\|g\|=1}(\sum_{i\in J}\|\Lambda_{ij}g\|^2)^{1/2}(\sum_{i\in J}\|g_i\|^2)^{1/2}\\
	&=\|T_{\Lambda}^{(j)}\|(\sum_{i\in J}\|g_i\|^2)^{1/2}\\
	&\le \sqrt{B_j}(\sum_{i\in J}\|g_i\|^2)^{1/2}
\end{align*}
for any finite subset $J\subset I$, $g_i\in\hs_i$,  then for $j\in[m]\setminus\{n\}$, 
we have
\begin{align*}
	&\|(T_{\Lambda}^{(n)}-T_{\Lambda}^{(j)})g_i\|\\
	&=\sup_{\|g\|=1}\left| \left\langle g,(T_{\Lambda}^{(n)}-T_{\Lambda}^{(j)})g_i\right\rangle\right|\\
	&=\sup_{\|g\|=1}\left| \left\langle g,\sum_{i\in J}(\Lambda_{in}^*-\Lambda_{ij}^*)g_i\right\rangle\right|\\
	&=\|\sum_{i\in J}(\Lambda_{in}^*-\Lambda_{ij}^*)g_i\|\\
	&\le \eta_j\|\sum_{i\in J}\Lambda_{in}^*g_i\|+\mu_j\|\sum_{i\in J}\Lambda_{ij}^*g_i\|+\lambda_j(\sum_{i\in J}\|g_i\|^2)^{1/2}\\
	&=\eta_j\sup_{\|g\|=1}|\left\langle g,T_{\Lambda}^{(n)}g_i\right\rangle |+\mu_j\sup_{\|g\|=1}|\left\langle g,T_{\Lambda}^{(j)}g_i\right\rangle |+\lambda_j(\sum_{i\in J}\|g_i\|^2)^{1/2}\\
	&\le \eta_j\|T_{\Lambda}^{(n)}\|(\sum_{i\in J}\|g_i\|^2)^{1/2}+\mu_j\|T_{\Lambda}^{(j)}\|(\sum_{i\in J}\|g_i\|^2)^{1/2}+\lambda_j(\sum_{i\in J}\|g_i\|^2)^{1/2}\\
	&\le (\lambda_j+\eta_j\sqrt{B_n}+\mu_j\sqrt{B_j})(\sum_{i\in J}\|g_i\|^2)^{1/2}.	
\end{align*}
This gives 
\begin{equation}
\|T_{\Lambda}^{(n)}-T_{\Lambda}^{(j)}\|\le \lambda_j+\eta_j\sqrt{B_n}+\mu_j\sqrt{B_j}.
\end{equation}
For $j\in[m]$ and $\sigma\subset I$, we define
$$T_{\Lambda}^{(j\sigma)}:\oplus_{i\in \sigma}\hs_i\longrightarrow\hs,~~T_{\Lambda}^{(j\sigma)}\{g_i\}=\sum_{i\in \sigma}\Lambda_{ij}^*g_i,~~~g_i\in\hs_i.$$
It is easy to see 
$$\|T_{\Lambda}^{(j\sigma)}g_i\|\le \|T_{\Lambda}^{(j)}g_i\|\le \sqrt{B_j}(\sum_{i\in J}\|g_i\|^2)^{1/2}.$$
Thus, $\|T_{\Lambda}^{(j\sigma)}\|\le \sqrt{B_j}$ for all $j\in[m]$. Similarly, by using (5.1) one can show that for any $j\in[m]\setminus\{n\}$,
$$\|T_{\Lambda}^{(n\sigma)}-T_{\Lambda}^{(j\sigma)}\|\le \lambda_j+\eta_j\sqrt{B_n}+\mu_j\sqrt{B_j}.$$
For any $f\in\hs$ and $j\in[m]\setminus\{n\}$, we have
\begin{align*}
	&\|(T_{\Lambda}^{(n\sigma)}(T_{\Lambda}^{(n\sigma)})^*-T_{\Lambda}^{(j\sigma)}(T_{\Lambda}^{(j\sigma)})^*)f\|\\
	&=\|(T_{\Lambda}^{(n\sigma)}(T_{\Lambda}^{(n\sigma)})^*-T_{\Lambda}^{(n\sigma)}(T_{\Lambda}^{(j\sigma)})^*+T_{\Lambda}^{(n\sigma)}(T_{\Lambda}^{(j\sigma)})^*-T_{\Lambda}^{(j\sigma)}(T_{\Lambda}^{(j\sigma)})^*)f\|\\
	&\le\|(T_{\Lambda}^{(n\sigma)}(T_{\Lambda}^{(n\sigma)})^*-T_{\Lambda}^{(n\sigma)}(T_{\Lambda}^{(j\sigma)})^*)f\|+\|(T_{\Lambda}^{(n\sigma)}(T_{\Lambda}^{(j\sigma)})^*-T_{\Lambda}^{(j\sigma)}(T_{\Lambda}^{(j\sigma)})^*)f\|\\
	&\le\|T_{\Lambda}^{(n\sigma)}\|\|((T_{\Lambda}^{(n\sigma)})^*-(T_{\Lambda}^{(j\sigma)})^*)f\|+\|(T_{\Lambda}^{(j\sigma)})^*\|\|(T_{\Lambda}^{(n\sigma)}-T_{\Lambda}^{(j\sigma)})f\|\\
	&\le (\lambda_j+\eta_j\sqrt{B_n}+\mu_j\sqrt{B_j})(\sqrt{B_n}+\sqrt{B_j})\|f\|.\tag{5.2}
\end{align*}
Let $\{\sigma_j\}_{j\in[m]}$ be any partition of $I$ and $T_{\Lambda}$ be the synthesis operator associated with the Bessel g-sequence $\{\Lambda_{ij}\}_{i\in \sigma_j,j\in[m]}$. By using (5.2), we have
\begin{align*}
	\|T^*_{\Lambda}f\|^2
	&=|\left\langle f,T_{\Lambda}T^*_{\Lambda}f\right\rangle |\\
	&=\bigg|\bigg\langle f,\sum_{i \in  I}\Lambda^*_{ij}\Lambda_{ij}f\bigg\rangle  \bigg| \\
	&=\bigg|\bigg\langle f,\sum_{i \in  \sigma_1}\Lambda^*_{i1}\Lambda_{i1}f+\cdots+\sum_{i \in  \sigma_n}\Lambda^*_{in}\Lambda_{in}f+\cdots+\sum_{i \in  \sigma_m}\Lambda^*_{im}\Lambda_{im}f\bigg\rangle  \bigg|\\
	&=\bigg|\bigg\langle f,\sum_{i \in  \sigma_1}\Lambda^*_{i1}\Lambda_{i1}f+\cdots+\sum_{j\in[m]}\sum_{i \in  \sigma_j}\Lambda^*_{in}\Lambda_{in}f\\
	&\qquad-\sum_{j\in[m]\setminus\{n\}}\sum_{i \in  \sigma_j}\Lambda^*_{in}\Lambda_{in}f+\cdots+\sum_{i \in  \sigma_m}\Lambda^*_{im}\Lambda_{im}f\bigg\rangle  \bigg|\\
	&=\bigg|\bigg\langle f,\sum_{i \in  I}\Lambda_{in}^*\Lambda_{in}f-\sum_{j\in[m]\setminus\{n\}}\sum_{i\in\sigma_j}(\Lambda_{in}^*\Lambda_{in}-\Lambda_{ij}^*\Lambda_{ij})f\bigg\rangle  \bigg|\\
	&\ge \bigg|\bigg\langle f,\sum_{i \in  I}\Lambda^*_{in}\Lambda_{in}f\bigg\rangle  \bigg|-\sum_{j\in[m]\setminus\{n\}}\bigg|\bigg\langle f,\sum_{i\in\sigma_j}(\Lambda_{in}^*\Lambda_{in}-\Lambda_{ij}^*\Lambda_{ij})f\bigg\rangle\bigg|\\
	&\ge\big|\big\langle f,T_{\Lambda}^{(n)}(T^{(n)}_{\Lambda})^*f\big\rangle \big|- \sum_{j\in[m]\setminus\{n\}}\|f\|\sup_{\|f_0\|=1}\bigg|\bigg\langle f_0,\sum_{i\in\sigma_j}(\Lambda_{in}^*\Lambda_{in}-\Lambda_{ij}^*\Lambda_{ij})f\bigg\rangle\bigg|\\
	&=\|(T^{(n)}_{\Lambda})^*f\|^2- \sum_{j\in[m]\setminus\{n\}}\|f\|\sup_{\|f_0\|=1}|\langle f_0, (T_{\Lambda}^{(n\sigma_j)}(T^{(n\sigma_j)}_{\Lambda})^*-T_{\Lambda}^{(j\sigma_j)}(T^{(j\sigma_j)}_{\Lambda})^*)f \rangle|\\
	&\ge A_n\|f\|^2-\sum_{j\in[m]\setminus\{n\}}\|f\|\|(T_{\Lambda}^{(n\sigma_j)}(T^{(n\sigma_j)}_{\Lambda})^*-T_{\Lambda}^{(j\sigma_j)}(T^{(j\sigma_j)}_{\Lambda})^*)f\|\\
	&\ge A_n\|f\|^2-\sum_{j\in[m]\setminus\{n\}}\|f\|(\lambda_j+\eta_j\sqrt{B_n}+\mu_j\sqrt{B_j})(\sqrt{B_n}+\sqrt{B_j})\|f\|\\
	&=(A_n-\sum_{j\in[m]\setminus\{n\}}(\lambda_j+\eta_j\sqrt{B_n}+\mu_j\sqrt{B_j})(\sqrt{B_n}+\sqrt{B_j}))\|f\|^2>0
\end{align*}
Hence, the $\{\Lambda_{ij}\}_{i\in \sigma_j,j\in[m]}$ is a g-frame for $\hs$ with required universal frame bounds. We complete the proof of the theorem.
\end{proof}
When the index $n$ in Theorem 5.1 is not fixed, we have the following result.
\begin{theorem}
For each $j\in[m]$, let $\Lambda_j=\{\Lambda_{ij}\}_{i\in I}$ be a g-frame for $\hs$ with frame bounds $A_j$ and $B_j$. Let $\lambda_j,~\eta_j,~\mu_j\ge 0$, $j\in[m-1]$ be such that
$$A=A_1-\sum_{j\in[m-1]}(\lambda_j+\eta_j\sqrt{B_j}+\mu_j\sqrt{B_{j+1}})(\sqrt{B_j}+\sqrt{B_{j+1}})>0$$
and 
$$\|\sum_{i\in J}(\Lambda_{ij}^*-\Lambda_{i(j+1)}^*)g_i\|\le \eta_j\|\sum_{i\in J}\Lambda_{ij}^*g_i\|+\mu_j\|\sum_{i\in J}\Lambda_{i(j+1)}^*g_i\|+\lambda_j(\sum_{i\in J}\|g_i\|^2)^{1/2}$$
for any finite subset $J\subset I$, $g_i\in\hs_i$ and $j\in[m-1]$. Then for any partition $\{\sigma_j\}_{j\in[m]}$ of $I$, the family $\{\Lambda_{ij}\}_{i\in \sigma_j,j\in[m]}$ is a g-frame for $\hs$ with universal frame bounds $A$ and $\sum_{j\in[m]}B_j$.
\end{theorem}
\begin{proof}
Clearly, $\sum_{j\in[m]}B_j$ is an upper universal canstant for the family $\{\Lambda_{ij}\}_{i\in \sigma_j,j\in[m]}$ for any partition $\{\sigma_j\}_{j\in[m]}$ of $I$. From the proof of Theorem 5.1, for any $j\in[m-1]$, we have
$$\|T_{\Lambda}^{(j\sigma)}-T_{\Lambda}^{((j+1)\sigma)}\|\le \lambda_j+\eta_j\sqrt{B_j}+\mu_j\sqrt{B_{j+1}}.$$
Furthermore
$$\|(T_{\Lambda}^{(j\sigma)}(T_{\Lambda}^{(j\sigma)})^*-T_{\Lambda}^{((j+1)\sigma)}(T_{\Lambda}^{((j+1)\sigma)})^*)\|\le (\lambda_j+\eta_j\sqrt{B_j}+\mu_j\sqrt{B_{j+1}})(\sqrt{B_j}+\sqrt{B_{j+1}})$$
for all $j\in[m-1]$. For all $f\in\hs$, we have
\begin{align*}
	\|T^*_{\Lambda}f\|^2&=|\left\langle f,T_{\Lambda}T^*_{\Lambda}f\right\rangle |\\
	&=\bigg|\bigg\langle f,\sum_{i \in  I}\Lambda^*_{ij}\Lambda_{ij}f\bigg\rangle  \bigg|\\
	&=\bigg|\bigg\langle f,\sum_{j\in[m]}\sum_{i \in  \sigma_j}\Lambda^*_{ij}\Lambda_{ij}f\bigg\rangle  \bigg| \\
	&=\bigg|\bigg\langle f,\sum_{j\in[m]}\big(\sum_{i \in  \sigma_j}\Lambda^*_{ij}\Lambda_{ij}f+\sum_{i \in  {I\setminus(\bigcup_{l\in[j]}\sigma_l)}}\Lambda^*_{ij}\Lambda_{ij}f-\sum_{i \in  {I\setminus(\bigcup_{l\in[j]}\sigma_l)}}\Lambda^*_{ij}\Lambda_{ij}f\big)\bigg\rangle  \bigg| \\
	&=\bigg|\bigg\langle f,\sum_{j\in[m]}\big(\sum_{i \in  {I\setminus(\bigcup_{l\in[j-1]}\sigma_l)}}\Lambda^*_{ij}\Lambda_{ij}f-\sum_{i \in  {I\setminus(\bigcup_{l\in[j]}\sigma_l)}}\Lambda^*_{ij}\Lambda_{ij}f\big)\bigg\rangle  \bigg|\\
	&=\bigg|\bigg\langle f,\sum_{i\in I}\Lambda_{i1}^*\Lambda_{i1}f-\sum_{j\in[m-1]}\big(\sum_{i \in  {I\setminus(\bigcup_{l\in[j]}\sigma_l)}}\Lambda^*_{ij}\Lambda_{ij}f\\
	&\qquad\qquad\qquad\qquad-\sum_{i \in  {I\setminus(\bigcup_{l\in[j]}\sigma_l)}}\Lambda^*_{i(j+1)}\Lambda_{i(j+1)}f\big)\bigg\rangle  \bigg|\\
	&\ge \bigg|\bigg\langle f,\sum_{i\in I}\Lambda_{i1}^*\Lambda_{i1}f\bigg\rangle  \bigg|-\sum_{j\in[m-1]}\bigg|\bigg\langle f,\sum_{i \in  {I\setminus(\bigcup_{l\in[j]}\sigma_l)}}\Lambda^*_{ij}\Lambda_{ij}f\\
	&\qquad\qquad\qquad\qquad-\sum_{i \in  {I\setminus(\bigcup_{l\in[j]}\sigma_l)}}\Lambda^*_{i(j+1)}\Lambda_{i(j+1)}f\bigg\rangle  \bigg|\\
	&\ge\left| \left\langle f, T_{\Lambda}^{(1)}(T_{\Lambda}^{(1)})^*f\right\rangle \right| -\sum_{j\in[m-1]}\|f\|\sup_{\|f_0\|=1}\bigg|\bigg\langle f_0,\sum_{i \in  {I\setminus(\bigcup_{l\in[j]}\sigma_l)}}\Lambda^*_{ij}\Lambda_{ij}f\\
	&\qquad\qquad\qquad\qquad-\sum_{i \in  {I\setminus(\bigcup_{l\in[j]}\sigma_l)}}\Lambda^*_{i(j+1)}\Lambda_{i(j+1)}f\bigg\rangle  \bigg|\\
	&\ge A_1\|f\|^2-\sum_{j\in[m-1]}\|f\|\sup_{\|f_0\|=1}\big|\big\langle f_0,(T_{\Lambda}^{(j(I\setminus(\bigcup_{l\in[j]}\sigma_l))}(T_{\Lambda}^{(j(I\setminus(\bigcup_{l\in[j]}\sigma_l))})^*\\
	&\qquad\qquad\qquad\qquad-T_{\Lambda}^{((j+1)(I\setminus(\bigcup_{l\in[j]}\sigma_l))}(T_{\Lambda}^{((j+1)(I\setminus(\bigcup_{l\in[j]}\sigma_l))})^*)f \big\rangle  \big|\\
	&\ge A_1\|f\|^2-\sum_{j\in[m-1]}\|f\|\big\|(T_{\Lambda}^{(j(I\setminus(\bigcup_{l\in[j]}\sigma_l))}(T_{\Lambda}^{(j(I\setminus(\bigcup_{l\in[j]}\sigma_l))})^*\\
	&\qquad\qquad\qquad\qquad-T_{\Lambda}^{((j+1)(I\setminus(\bigcup_{l\in[j]}\sigma_l))}(T_{\Lambda}^{((j+1)(I\setminus(\bigcup_{l\in[j]}\sigma_l))})^*)f  \big\|\\
	&\ge A_1\|f\|^2-\sum_{j\in[m-1]}\|f\|^2\big\|T_{\Lambda}^{(j(I\setminus(\bigcup_{l\in[j]}\sigma_l))}(T_{\Lambda}^{(j(I\setminus(\bigcup_{l\in[j]}\sigma_l))})^*\\
	&\qquad\qquad\qquad\qquad-T_{\Lambda}^{((j+1)(I\setminus(\bigcup_{l\in[j]}\sigma_l))}(T_{\Lambda}^{((j+1)(I\setminus(\bigcup_{l\in[j]}\sigma_l))})^*  \big\|\\
	&\ge\big[A_1-\sum_{j\in[m-1]}(\lambda_j+\eta_j\sqrt{B_j}+\mu_j\sqrt{B_{j+1}})(\sqrt{B_j}+\sqrt{B_{j+1}})\big]\|f\|^2.
\end{align*}
This gives the lower universal frame bound. We complete the proof.
\end{proof}
Finally, we consider the stability of g-frame with a finite number of bounded, invertible operator. When $T_i=T_j$ for all $i,j\in I$, the Proposition 6.2 of \cite{bemrose2015weaving} can be obtained from the following proposition.
\begin{proposition}
Let $\{\Lambda_i\}_{i\in I}$ be a g-frame for $\hs$ with frame bounds $A$ and $B$ and let $T_i$ be a bounded, invertible operator for all $i\in I$. If
$$\|I_{\hs}-T_i\|^2<\frac{A}{B},$$
then $\{\Lambda_i\}_{i\in I}$ and $\{\Lambda_iT_i\}_{i\in I}$ are woven.
\end{proposition}
\begin{proof}
Note that $T_j$ is invertible and thus $\{\Lambda_iT_i\}_{i\in I}$ is automatically a g-frame. It is easy to compute that $(1+\|T_i\|^2)B$ is an upper frame bound of $\{\Lambda_i\}_{i\in\sigma}\cup\{\Lambda_iT_i\}_{i\in\sigma^c}$.
For every $\sigma\in I$ and for every $\in\hs$ we have by Minkowski's inequality and subadditivity of the square root function
\begin{align*}
	&\big(\sum_{i\in\sigma}\|\Lambda_if\|^2+\sum_{i\in\sigma^c}\|\Lambda_iT_if\|^2\big)^{1/2}\\
	&\qquad=\big(\sum_{i\in\sigma}\|\Lambda_if\|^2+\sum_{i\in\sigma^c}\|\Lambda_i(f-(f-T_if))\|^2\big)^{1/2}\\
	&\qquad=\big(\sum_{i\in\sigma}\|\Lambda_if\|^2+\sum_{i\in\sigma^c}\|\Lambda_if-\Lambda_i(I_{\hs}-T_i)f\|^2\big)^{1/2}\\
	&\qquad\ge \big(\sum_{i\in\sigma}\|\Lambda_if\|^2+\sum_{i\in\sigma^c}\|\Lambda_if\|^2-\sum_{i\in\sigma^c}\|\Lambda_i(I_{\hs}-T_i)f\|^2\big)^{1/2}\\
	&\qquad\ge \big(\sum_{i\in I}\|\Lambda_if\|^2\big)^{1/2}-\big(\sum_{i\in\sigma^c}\|\Lambda_i(I_{\hs}-T_i)f\|^2\big)^{1/2}\\
	&\qquad\ge \sqrt{A}\|f\|-\sqrt{B}\|(I_{\hs}-T_i)f\|\\
	&\qquad\ge (\sqrt{A}-\sqrt{B}\|I_{\hs}-T_i\|)\|f\|
\end{align*}
Thus, $\{\Lambda_i\}_{i\in\sigma}\cup\{\Lambda_iT_i\}_{i\in\sigma^c}$ forms a g-frame having 
$$A-B\|I_{\hs}-T_i\|^2>0$$
as its lower frame bound.
\end{proof}
\begin{corollary}
Let $\{\Lambda_i\}_{i\in I}$ be a g-frame for $\hs$ with frame bounds $A$ and $B$ and frame operator $S_{\Lambda}$. If $B/A<2$, then $\Lambda$ and the scaled canonical dual g-frame $\widetilde{\Lambda}=\{\frac{2AB}{A+B}\Lambda_iS_{\Lambda}^{-1}\}_{i\in I}$ are woven.
\end{corollary}
\begin{proof}
We apply Proposition 5.3 to the operator $T=T_i=T_j=\frac{2AB}{A+B}S_{\Lambda}^{-1}$ for all $i,j\in I$. Since the spectrum of $S_{\Lambda}$ is contained in the interval $[A,B]$, the spectrum of $I_{\hs}-T$ is contained in the interval $[\frac{A-B}{A+B},\frac{B-A}{A+B}]$ and thus $$\|I_{\hs}-T\|\le \frac{B-A}{B+A}.$$
This norm is majorized by $\sqrt{(A/B)}$, whenever $B/A\le 2$.
\end{proof}
\section*{Acknowledgements}
The research is supported by the National Natural Science Foundation of China (Nos. 11271001 and 61370147), and the Fundamental Research Funds for the Central Universities (No. ZYGX2016KYQD143).

\bibliographystyle{amsplain}

\bibliography{my} 

\end{document}